\documentclass{lms}

\usepackage{amsfonts, amsmath, amssymb}
\usepackage{mathrsfs}
\usepackage{xypic}
\usepackage{tikz-cd}

\providecommand{\inj}{\mathop{inj}}
\providecommand{\st}{\,\big|\,}
\providecommand{\real}{\mathbb{R}}
\providecommand{\nat}{\mathbb{N}}
\providecommand{\tor}{\mathbb{T}}
\providecommand{\vol}{\mathop{vol}}
\providecommand{\trace}{\mathop{trace}}

\newtheorem{theorem}{Theorem}[section]
\newtheorem{lemma}[theorem]{Lemma}

\newnumbered{remark}{Remark}

\title[Totally geodesic maps into no focal points]{Totally geodesic maps into manifolds with no focal points}

\author{James Dibble}

\classno{53C21, 53C43 (primary), 53C22, 53C44, 58J35 (secondary)}

\begin{document}
\maketitle

\begin{abstract}
The set of totally geodesic representatives of a homotopy class of maps from a compact Riemannian manifold $M$ with nonnegative Ricci curvature into a complete Riemannian manifold $N$ with no focal points is path-connected and, when nonempty, equal to the set of energy-minimizing maps in that class. When $N$ is compact, each map from a product $W \times M$ into $N$ is homotopic to a map that's totally geodesic on each $M$-fiber. These results may be used to extend to the case of no focal points a number of splitting theorems of Cao--Cheeger--Rong about manifolds with nonpositive sectional curvature and, in turn, to generalize a non-collapsing theorem of Heintze--Margulis. In contrast with previous approaches, they are proved using neither a geometric flow nor the Bochner identity for harmonic maps.
\end{abstract}

\section{Introduction}

A map $u : M \rightarrow N$ between Riemannian manifolds is \textbf{totally geodesic} if, whenever $\gamma : (a,b) \to M$ is a geodesic, the composition $u \circ \gamma$ is a geodesic. A celebrated theorem of Eells--Sampson \cite{EellsSampson1964} states that, when $M$ is compact with nonnegative Ricci curvature and $N$ is compact with nonpositive sectional curvature, each map from $M$ to $N$ is homotopic to a totally geodesic map. They proved this by showing that unique solutions to their heat equation \eqref{heat equation} exist for all time and uniformly subconverge to harmonic maps, which, by the Bochner identity for harmonic maps \eqref{bochner identity}, must be totally geodesic. Hartman \cite{Hartman1967} further showed that, when two maps are initially close, the distance between the corresponding solutions to the heat equation is nonincreasing. It follows that the set of totally geodesic maps in each homotopy class is path-connected, that all harmonic maps are energy-minimizing and totally geodesic, and, as noted by Cao--Cheeger--Rong \cite{CaoCheegerRong2004}, that every map from a product $W \times M$ into $N$ is homotopic to a map that's totally geodesic on each $M$-fiber.

The purpose of this paper is to generalize those results to codomains with no focal points, a synthetic condition more general than having nonpositive curvature.

\begin{theorem}\label{main theorem for nnrc}
    Let $M$ be a compact Riemannian manifold with nonnegative Ricci curvature, $N$ a complete Riemannian manifold with no focal points, and $[F]$ a homotopy class of maps from $M$ to $N$. Then the following hold:\\
    \textbf{(a)} The set of totally geodesic maps in $[F]$ is path-connected;\\
    \textbf{(b)} If $[F]$ contains a totally geodesic map, then a map in $[F]$ is energy-minimizing if and only if it is totally geodesic;\\
    \textbf{(c)} If $N$ is compact, then $[F]$ contains a totally geodesic map.
\end{theorem}

\begin{theorem}\label{product version for nnrc}
    Let $M$ be a compact Riemannian manifold with nonnegative Ricci curvature, $N$ a compact Riemannian manifold with no focal points, and $W$ a manifold. Then each map from $W \times M$ into $N$ is homotopic to a map that's totally geodesic on each $M$-fiber.
\end{theorem}

\noindent In fact, somewhat more general versions of the above will be proved. These hold for domains that are finitely covered by a diffeomorphic product in a commutative diagram inspired by the Cheeger--Gromoll splitting theorem.

The methods of this paper are novel in that they use neither a geometric flow nor the Bochner identity for harmonic maps \eqref{bochner identity}. Instead, the key tools are the Cheeger--Gromoll splitting theorem, the Riemannian center of mass, the flat torus theorem, and the integral techniques in \cite{Dibble2018a}. Thus Theorem \ref{main theorem for nnrc} may be considered a spiritual converse to a theorem of Jost \cite{Jost1991}, who used the results of Eells--Sampson to prove the flat torus theorem for manifolds with nonpositive curvature. Moreover, as Eschenburg--Heintze \cite{EschenburgHeintze1984} gave a proof of the Cheeger--Gromoll splitting theorem that uses only the maximum principle and a Bochner--Lichnerowicz formula, those are the only PDE results required.

Theorems \ref{main theorem for nnrc} and \ref{product version for nnrc} allow one to generalize to manifolds with no focal points results for nonpositively curved manifolds that depend only on energy-minimizing maps being totally geodesic. For example, Theorem \ref{product version for nnrc} may be used to extend to the case of no focal points a number of splitting theorems of Cao--Cheeger--Rong \cite{CaoCheegerRong2001,CaoCheegerRong2004}, including the following.

\begin{theorem}\label{cao--cheeger--rong}
    Let $M$ and $N$ be compact manifolds of the same dimension. Suppose $M$ admits an $F$-structure. If there exists a continuous function $f : M \to N$ with nonzero degree, then every metric on $N$ with no focal points admits a local splitting structure for which there is a consistency map homotopic to $f$.
\end{theorem}

\noindent Loosely speaking, a manifold admits an $F$-structure if it can be cut into pieces that, up to finite covers, admit effective torus actions that are compatible on overlaps; precise definitions may be found in \cite{CaoCheegerRong2004}. It follows that the universal cover of any compact manifold with no focal points that admits an $F$-structure is the union of isometric products $D_i \times \real^k$, each of which is convex, whose Euclidean factors project to immersed submanifolds that, up to homotopy, contain the orbits of the torus actions. With this, one may extend a non-collapsing theorem of Heintze--Margulis.

\begin{theorem}\label{heintze--margulis extension}
    For each $n \in \nat$, there exists $\varepsilon = \varepsilon(n) > 0$ such that, if $M$ is a compact $n$-dimensional manifold that admits a Riemannian metric with no focal points and negative Ricci curvature at a point, then, for every metric on $M$ with $|\mathop{sec}_M| \leq 1$, there is a point at which the injectivity radius is at least $\varepsilon$.
\end{theorem}

\noindent The proofs of Theorems \ref{cao--cheeger--rong} and \ref{heintze--margulis extension} will be omitted, as they follow the arguments in \cite{CaoCheegerRong2004} almost verbatim, once Theorem \ref{product version for nnrc} is used in place of a parameterized heat flow and the results of Eells--Sampson and Hartman.

\subsection*{Organization of the paper}

Section 2 discusses the heat flow and other previously known methods for proving the existence of energy-minimizing or totally geodesic maps into manifolds with nonpositive curvature. Those results are included only to place the rest of the paper into context; they are not needed to prove the main theorems. The reader unfamiliar with manifolds with no focal points is encouraged to read subsection 3.2 first.

Section 3 provides essential background information. Subsection 3.1 defines the notions of convexity used throughout the paper. Subsection 3.2 discusses basic properties of manifolds with no conjugate points and no focal points. Subsection 3.3 describes a method for constructing a homotopy between two maps from a torus into a manifold with no conjugate points that are compatible at the level of fundamental group. Subsection 3.4 states basic properties of the Riemannian center of mass. Subsection 3.5 details the class of domains to which the main theorems apply.

Section 4 contains the statement and proof of the main theorems. Theorems \ref{main theorem for nnrc} and \ref{product version for nnrc} are obtained as immediate consequences.

\section{Heat flow methods}

This section discusses previous approaches to the problem of finding energy-minimizing or totally geodesic representatives of homotopy classes of maps into manifolds with nonpositive curvature. Most of this is well known and may be found in a variety of sources, including \cite{Jost2011}, \cite{Xin1996}, and \cite{EellsSampson1964}.

Let $u : M \to N$ be a $C^1$ map between Riemannian manifolds. Denote by $u^{-1}(TN) \to M$ the vector bundle that attaches to each $x \in M$ the tangent space $T_{u(x)} N$ and by $\mathscr{L} \big( TM, u^{-1}(TN) \big) \to M$ the vector bundle that attaches to each $x \in M$ the space of linear maps $T_x M \to T_{u(x)} N$. The \textbf{differential} $du$ is the section of $\mathscr{L} \big( TM, u^{-1}(TN) \big)$ satisfying $du(v) = u_*(v)$ for each $v \in TM$, where $u_* : TM \to TN$ is the push-forward map. Denote by $\mathscr{L} \big( TM \odot TM, u^{-1}(TN) \big) \to M$ the vector bundle that attaches to each $x \in M$ the space of linear maps $T_x M \odot T_x M \to T_{u(x)} N$, where $\odot$ is the symmetric product. With respect to the natural bundle metrics $g^{-1}$ on $T^* M$ and $h$ on $u^{-1}(TN)$, the bundle $\mathscr{L} \big( TM, u^{-1}(TN) \big) \cong T^* M \otimes u^{-1}(TN)$ is endowed with a natural connection. When $u$ is $C^2$, its \textbf{second fundamental form} is the section of $\mathscr{L} \big( TM \odot TM, u^{-1}(TN) \big)$ defined by $B_u = \nabla du$. If $\gamma : (a,b) \rightarrow M$ is a geodesic, then $B_u(\gamma',\gamma') = \nabla_{(u \circ \gamma)'} (u \circ \gamma)'$. Therefore, a $C^2$ map is totally geodesic if and only if its second fundamental form vanishes.

The \textbf{energy density} of a $C^1$ map $u : M \rightarrow N$ is the function $e_u : M \rightarrow [0,\infty)$ defined by $e_u(x) = \frac{1}{2}\| d_x u \|^2$, where $\| \cdot \|$ is the norm induced by the metric on $\mathscr{L} \big( TM, u^{-1}(TN) \big)$. The \textbf{energy} of $u$ is $E(u) = \int_M e_u d{\vol}_M$. When $u$ is $C^2$, its \textbf{tension field} is the section of $u^{-1}(TN)$ defined by $\tau_u = \trace(B_u)$. A $C^2$ map is \textbf{harmonic} if its tension field vanishes. Totally geodesic maps are harmonic, but the converse is false. For any $C^1$ variation $F : [0,\varepsilon) \times M \rightarrow N$ of $u$, the first variation of energy is
\[
    \frac{d}{dt}\Big|_{t=0} E\big( F(t,\cdot) \big) = -\int_M h \big( \tau_u,V \big) d{\vol}_M\textrm{,}
\]
where $V$ is the variation field of $F$ at time $t = 0$. In other words, $\tau_u$ is the negative gradient field of the energy functional, and harmonic maps are critical points of it.

For compact $M$, unique short-term solutions to the Eells--Sampson \cite{EellsSampson1964} heat equation
\begin{equation}\label{heat equation}
    \begin{aligned}
        &\frac{\partial u}{\partial t} = \tau_u & \textrm{ on } & \phantom{aa} M \times [0,\varepsilon)\\
        &u = u_0 & \textrm{ on } & \phantom{aa} M \times \{ 0 \}
    \end{aligned}
\end{equation}
exist for any $C^1$ initial data $u_0 : M \rightarrow N$. When $N$ is compact with nonpositive sectional curvature, solutions exist for all time and uniformly subconverge to harmonic maps. It follows that each homotopy class of maps contains an energy-minimizing harmonic representative.

When $M$ has nonnegative Ricci curvature and $N$ nonpositive sectional curvature, integrating both sides of the following identity, which dates to the work of Bochner \cite{Bochner1940}, implies that harmonic maps are totally geodesic:
\begin{equation}\label{bochner identity}
    \Delta e_u = \| B_u \|^2 + h \big( df \big( \mathop{Ric}\phantom{}^M\!(e_i) \big), df \big( e_i \big) \big) - \mathop{Rm}\!\phantom{}^N\!\big( df(e_i), df(e_j), df(e_j), df(e_i) \big)\textrm{.}
\end{equation}
Here, $u$ is assumed to be harmonic, and $\Delta$ is the Laplace--Beltrami operator; that is, $\Delta e_u = \mathop{div}(\nabla e_u)$.

Hartman \cite{Hartman1967} showed that, when a homotopy class $[F]$ of maps into $N$ contains a harmonic representative, the heat flow within $[F]$ uniformly converges to a harmonic map, the space of harmonic maps in $[F]$ is path-connected, and a map in $[F]$ is harmonic if and only if it is energy-minimizing. He derived these results from the following monotonicity: If $u,v : M \times [0,\infty) \to N$ are solutions to \eqref{heat equation} whose respective initial data $u_0,v_0 : M \to N$ are sufficiently close, then the function $t \mapsto \max_{x \in M} d_N \big( u(x,t), v(x,t) \big)$ is nonincreasing. Cao--Cheeger--Rong \cite{CaoCheegerRong2004} used this monotonicity to show that the parameterized heat flow obtained by simultaneously flowing the restriction of any map $f : W \times M \to N$ to the various $M$-fibers uniformly converges to a map that's totally geodesic on each $M$-fiber.

A theorem of Li--Zhu \cite{LiZhu2010} implies that, when $N$ has no focal points, solutions to the heat equation exist for all time and subconverge to harmonic maps. It follows that each homotopy class of maps from $M$ to $N$ contains an energy-minimizing representative, a result which was also proved earlier by Xin \cite{Xin1984} using the direct method and the maximum principle. However, without curvature assumptions, the Bochner identity \eqref{bochner identity} cannot be gainfully employed. Thus it's not clear that, for domains $M$ with nonnegative Ricci curvature, energy-minimizing maps into $N$ are totally geodesic. Moreover, without Hartman's monotonicity, it's not clear that the long-term existence and uniform subconvergence on each $M$-fiber of a solution to the parameterized heat flow implies that the flow converges to a continuous map on $W \times M$. Those are the main technical points bypassed in this paper.

\section{Preliminaries}

This section contains background information on the tools used in the proof of the main theorems. Most of these results appear elsewhere in the literature; notable exceptions are Theorem \ref{structure theorem for geodesic loops} and Lemma \ref{normal covering}. Throughout the rest of the paper, $(M,g)$ and $(N,h)$ will be complete Riemannian manifolds. For the sake of simplicity, they will always assumed to be $C^2$, although many of the results hold when $M$ is only $C^1$. Where appropriate, $\iota$ denotes inclusion. That is, for a map $f$ defined on $A \times B$ and $(p,z) \in A \times B$, $f \circ \iota_z(p) = f(p,z) = f \circ \iota_p(z)$. Certain projections are denoted by $\rho_i$ and $\pi_i$. The symbol $\rho$ without a subscript denotes the strong convexity radius, and $\pi : \hat{N} \to N$ denotes the universal covering map. Note that the symbol $\pi_1$ is overloaded, as it can refer to both a projection and the fundamental group of a manifold.

\subsection{Convexity}

A subset $X$ of $N$ is \textbf{convex} if any two points in $X$ are joined by a minimal geodesic that lies inside $X$ and \textbf{strongly convex} if any two points in $X$ are joined by a unique minimal geodesic in $N$ and all such geodesics lie inside $X$. The \textbf{convex hull} of a set $Y$, denoted $\mathop{conv}(Y)$, is the smallest set that contains $Y$ and has the following property: If $x,y \in \mathop{conv}(Y)$, then $\mathop{conv}(Y)$ contains every minimal geodesic from $x$ to $y$. That is, $\mathop{conv}(Y)$ is the intersection of all sets that contain $Y$ and have that property. If $Y$ is contained in a strongly convex set, then $\mathop{conv}(Y)$ is the intersection of all strongly convex sets containing $Y$. The \textbf{convex closure} of $Y$ is the convex hull of the closure of $Y$.

A function $f : I \to [-\infty,\infty]$, where $I \subseteq \real$ is an interval, is \textbf{strictly convex} if $f(st_1 + (1-s)t_2) < sf(t_1) + (1-s)f(t_2)$ for all $s \in (0,1)$ and all distinct $t_1,t_2 \in I$. A function $f : N \to [-\infty,\infty]$ is \textbf{strictly convex} if its restriction to any nonconstant geodesic is strictly convex. A $C^2$ function $f$ is strictly convex if and only if its Hessian $\nabla^2 f$ is positive definite or, equivalently, $(f \circ \gamma)'' > 0$ for all nonconstant geodesics $\gamma$ along which $f$ is defined. If $f$ is a strictly convex function and $Y$ is a convex subset of the domain of $f$, then $f$ has at most one local minimum in $Y$, which, if there is one, must be a global minimum. If $(\Lambda,\mu)$ is a measure space and $f_\lambda$ is family of strictly convex functions indexed over $\Lambda$ and with a common domain, then $x \mapsto \int_\Lambda f_\lambda(x) \,d\mu$ is strictly convex.

For each $p \in N$, the \textbf{injectivity radius at $p$} is
\[
    \inj(p) = \max \big\{ R > 0 \st \exp_p|_{B(0,s)} \textrm{ is injective for all } 0 < s < R \big\}\textrm{,}
\]
the \textbf{convexity radius at $p$} is
\[
    r(p) = \textrm{max} \{ R > 0 \st B(p,s) \textrm{ is strongly convex for all } 0 < s < R \}\textrm{,}
\]
and the \textbf{strong convexity radius at $p$} is
\[
    \rho(p) = \textrm{max} \{ R > 0 \st r(x) \leq R \textrm{ for all } x \in B(p,R) \}\textrm{.}
\]
Each of these may be infinite. For a subset $X$ of $N$, one writes $\inj(X) = \inf_{x \in X} \inj(x)$, $r(X) = \inf_{x \in X} r(x)$, and $\rho(X) = \inf_{x \in X} \rho(x)$.

It is well known that $\inj$, $r$, and $\rho$ satisfy $0 < \rho \leq r \leq \inj$ \cite{Klingenberg1995}. By taking the second variation of arclength, one may show that, for all $p \in N$, $d(\cdot,p)$ is strictly convex on $B(p,r(p)) \setminus \{ p \}$. It follows that $d^2(\cdot,p)$ is strictly convex on $B(p,r(p))$.

\subsection{Manifolds with no conjugate points or no focal points}

If $\gamma : [a,b] \to N$ is a geodesic, then $\gamma(a)$ and $\gamma(b)$ are \textbf{conjugate along $\gamma$} if there exists a nontrivial Jacobi field $J$ along $\gamma$ that vanishes at $a$ and $b$. If $S$ is a submanifold of $N$ and $\gamma : [a,b] \to N$ is a geodesic orthogonal to $S$ at $\gamma(a)$, then an \textbf{$S$-Jacobi field along $\gamma$} is the variation field of a variation of $\gamma$ through geodesics that are initially perpendicular to $S$, and $\gamma(b)$ is \textbf{focal to $S$ along $\gamma$} if there exists a nontrivial $S$-Jacobi field along $\gamma$ that vanishes at $b$. One says that $N$ has \textbf{no conjugate points} if no two points are conjugate along any geodesic connecting them and \textbf{no focal points} if no point is focal to a totally geodesic submanifold along any geodesic connecting them. If $\pi : \hat{N} \to N$ is the Riemannian universal covering space of $N$, then $N$ has no conjugate points or no focal points if and only if $\hat{N}$ has the same property. A refinement of the classical Cartan--Hadamard theorem states that manifolds with no focal points, as a class, lie between those with nonpositive curvature and those with no conjugate points. This may be seen from the following well-known result \cite{O'Sullivan1974}.

\begin{lemma}[(O'Sullivan)]\label{jacobi fields}
    Let $N$ be a complete Riemannian manifold. Then the following hold:\\
    \textbf{(a)} $N$ has nonpositive sectional curvature if and only if $\frac{d^2}{dt^2} \|J\|^2 \geq 0$ for all Jacobi fields $J$ along geodesics $\gamma : [0,\varepsilon) \to N$ and all $t$;\\
    \textbf{(b)} $N$ has no focal points if and only if $\frac{d}{dt} \|J\|^2 > 0$ for all nontrivial Jacobi fields $J$ along geodesics $\gamma : [0,\varepsilon) \to N$ satisfying $J(0) = 0$ and all $t > 0$;\\
    \textbf{(c)} $N$ has no conjugate points if and only if $\|J\|^2 > 0$ for all nontrivial Jacobi fields $J$ along geodesics $\gamma : [0,\varepsilon) \to N$ satisfying $J(0) = 0$ and all $t > 0$.
\end{lemma}

\noindent It follows that $N$ has nonpositive curvature if and only if the distance function $d(\cdot,\cdot) : \hat{N} \times \hat{N} \to [0,\infty)$ is convex, no focal points if and only if $d(\cdot,\hat{p})$ is convex for each $\hat{p} \in \hat{N}$, and no conjugate points if and only if $d(\cdot,\hat{p})$ is nonsingular for all $\hat{p} \in \hat{N}$. In particular, $N$ has no focal points if and only if every open ball $B(\hat{p},R) \subseteq \hat{N}$ is strongly convex. In terms of geometric radiuses, $N$ has no conjugate points if and only if $\inj(\hat{N}) = \infty$, while $N$ has no focal points if and only if $r(\hat{N}) = \infty$ or, equivalently, $\rho(\hat{N}) = \infty$. Note that Gulliver \cite{Gulliver1975} constructed examples of compact manifolds with no conjugate points but focal points and, respectively, no focal points but some positive curvature.

It is a result of Hermann \cite{Hermann1963} that, whenever $S$ is a closed and connected submanifold of $N$ to which no point of $N$ is focal, the restriction of the exponential map to the normal bundle $S^\perp$ is a covering map. In particular, when $N$ has no conjugate points, $\exp_{\hat{p}} : T_{\hat{p}} \hat{N} \to \hat{N}$ is a diffeomorphism for each $\hat{p} \in \hat{N}$, so $\pi_1(N)$ is torsion free and, moreover, $N$ is an Eilenberg--Mac Lane space $K(\pi_1(N),1)$. The following is an early result of Busemann \cite{Busemann1955} about manifolds with no conjugate points.

\begin{lemma}[(Busemann)]\label{closed geodesics have the same length}
    Let $N$ be a complete Riemannian manifold with no conjugate points. Then any closed geodesic in $N$ minimizes length in its free homotopy class.
\end{lemma}

\noindent The proof of Lemma \ref{closed geodesics have the same length} is elegant and demonstrates the synthetic nature of many arguments in the area: If there were a loop in $N$ shorter than a closed geodesic in its free homotopy class, then one could iterate both a sufficiently large number of times, lift a homotopy connecting those iterates to $\hat{N}$, and thereby produce a path connecting two points shorter than the unique geodesic between them.

The \textbf{displacement function} of an isometry $\phi$ of $\hat{N}$ is the map $x \mapsto d(x,\phi(x))$, and the \textbf{minimum set} of $\phi$, denoted $\min(\phi)$, is the minimum set of its displacement function. An \textbf{axis} of $\phi$ is a geodesic $\gamma : \real \to \hat{N}$ such that, for some $t_0$, $\phi(\gamma(t)) = \gamma(t + t_0)$ for all $t \in \real$. Canonically associated to each $g \in \pi_1(N)$ is an isometric deck transformation $\phi_g$ of $\hat{N}$. An \textbf{axis} of $g$ is an axis of $\phi_g$, and the \textbf{minimum set} of $g$, denoted $\min(g)$, is the minimum set of $\phi_g$. The axes of $g$ are exactly the lifts of closed geodesics in the free homotopy class determined by $g$, and, by Lemma \ref{closed geodesics have the same length}, $\min(g)$ is the union of the axes of $g$. If $G$ is a subgroup of $\pi_1(N)$, then $\min(G) = \cap_{g \in G} \min(g)$. The following theorem of O'Sullivan \cite{O'Sullivan1976}, known as the flat torus theorem, generalizes a result of Gromoll--Wolf \cite{GromollWolf1971} and, independently, Lawson--Yau \cite{LawsonYau1972} for nonpositively curved manifolds.

\begin{theorem}[(O'Sullivan)]\label{flat torus theorem}
    Let $N$ be a complete Riemannian manifold with no focal points and $G$ an Abelian subgroup of $\pi_1(N)$ of rank $m$. Then the following hold:

    \vspace{2pt}

    \noindent \textbf{(a)} $\min(G)$ is isometric to $D \times \real^m$ for some closed and strongly convex subset $D$ of $\hat{N}$;

    \vspace{2pt}

    \noindent \textbf{(b)} $G$ acts by translation on the $\real^m$-fibers of $\min(G) \cong D \times \real^m$;

    \vspace{2pt}

    \noindent \textbf{(c)} The restriction of $\pi : \hat{N} \to N$ to each of the $\real^m$-fibers of $\min{G} \cong D \times \real^m$ descends to an isometric and totally geodesic immersion from a flat torus $\tor^m$ into $N$ whose induced homomorphism has, up to path-conjugation, image $G$;

    \vspace{2pt}

    \noindent \textbf{(d)} If $N$ is compact, then $D$ is nonempty.
\end{theorem}

\noindent It is known that the flat torus theorem may fail to hold for compact manifolds with no conjugate points \cite{KleinerUnpublished}, in that not every Abelian subgroup of $\pi_1(N)$ need be represented by a totally geodesic flat torus.

\subsection{Spaces of geodesic loops}

Since a complete manifold with no conjugate points is an Eilenberg--Mac Lane space, maps into it are determined up to homotopy by what they do at the level of fundamental group. In particular, one has the following.

\begin{lemma}\label{maps determined by fundamental group}
    Let $N$ be a complete Riemannian manifold with no conjugate points, $f,g : \tor^k \to N$ continuous maps, $\theta \in \tor^k$ a basepoint for $\pi_1(\tor^k)$, and $\{[s_1],\ldots,[s_k]\}$ a minimal generating set for $\pi_1(\tor^k)$. Then $f$ is homotopic to $g$ if and only if $f_*([s_1],\ldots,[s_k]) \cong g_*([s_1],\ldots,[s_k])$, in the sense that there exists a path $\alpha$ from $f(\theta)$ to $g(\theta)$ such that, for each $i$, conjugating $[f \circ s_i]$ by $\alpha$ yields $[g \circ s_i]$.
\end{lemma}

\noindent Lemma \ref{maps determined by fundamental group} may also be proved more constructively. This second approach will help clarify the proof of Theorem \ref{main theorem for nnrc}(a).

Denote by $T^k N \xrightarrow{\pi} N$ the tensor bundle that attaches to each $y \in N$ the set $T^k_y = \{ v_1 \otimes \cdots \otimes v_k \st v_i \in T_y N \}$ of $(k,0)$-tensors at $y$. Let
\[
    N_k = \{ v_1 \otimes \cdots \otimes v_k \in T^k N \st \exp(v_1) = \cdots = \exp(v_k) \}\mathop{.}
\]
That is, $N_k$ consists of the $k$-tensors whose components are the initial vectors of geodesic loops. Since $0_y \otimes \cdots \otimes 0_y \in N_k$ for all $y \in N$, $N_k$ is nonempty. Define $f : T^k N \to N^{k+1}$ by
\[
    f(v_1 \otimes \cdots \otimes v_k) = (\pi(v_1 \otimes \cdots \otimes v_k),\exp(v_1),\ldots,\exp(v_k))\textrm{.}
\]
When $N$ has no conjugate points, the derivative of $\exp_{\pi(v)}$ is nonsingular at each $v \in TN$; that is, $D_v \exp_{\pi(v)} : T_v(T_{\pi(v)} N) \to T_{\exp_v} N$ is a linear isomorphism. It follows that $f$ has constant rank $(k+1)\dim(N)$. Denote by $D = \{ (y,\ldots,y) \in N^{k+1} \st y \in N \}$ the diagonal in $N^{k+1}$. Then the inverse function theorem implies that $N_k = f^{-1}(D)$ is an embedded $C^1$ submanifold of $T^k N$ of dimension $\dim(N)$.

For a fixed $v = v_1 \otimes \cdots \otimes v_k \in N_k$, let $w = (0,w_1,\ldots,w_k) \in T_v N_k \subseteq T_v (T^k N) \cong T_{\pi(v)} N \times T_{v_1}(T_{\pi(v)} N) \times \cdots \times T_{v_k}(T_{\pi(v)} N)$. By the definition of $N_k$,
\[
    D_{v_i} \exp_{\pi(v_1 \otimes \cdots \otimes v_k)}(w_i) = D_{v_1 \otimes \cdots \otimes v_k} \pi|_{N_k}(w) = 0
\]
for all $i$. It follows that the restriction of $\pi$ to $N_k$ has constant rank $\dim(N)$ and, consequently, $\pi|_{N_k} : N_k \to N$ is a local diffeomorphism.

Endow $N_k$ with the pull-back metric from $\pi|_{N_k}$, and define $C^1$ length functions $L_i : N_k \to [0,\infty)$ by $L_i(v_1 \otimes \cdots \otimes v_k) = \|v_i\|$. By the first variation formula, the gradient of each $L_i$ satisfies $\| \nabla L_i \| < 2$. Thus $N_k$ is complete and, consequently, $\pi|_{N_k}$ is a covering map.

In what follows, a basepoint $p \in N$ will be assumed for $\pi_1(N)$. If $[\sigma] \in \pi_1(N)$, then, since $N$ has no conjugate points, there exists a unique $v \in T_p N$ such that the map $\gamma_v(t) = \exp(tv)$, for $t \in [0,1]$, is a geodesic loop in $[\sigma]$. For any $\Sigma = ([\sigma_1],\ldots,[\sigma_k]) \in \pi_1(N)^k$, there exists a unique $u_1 \otimes \cdots \otimes u_k \in N_k$ such that each $u_i$ is, in this way, the initial vector of a geodesic loop $\gamma_{u_i} \in [\sigma_i]$. Denote by $N_\Sigma$ the connected component of $N_k$ containing $u_1 \otimes \cdots \otimes u_k$ and by $\pi_\Sigma$ the restriction of $\pi$ to $N_\Sigma$.

\begin{theorem}\label{structure theorem for geodesic loops}
    Let $N$ be a complete and connected Riemannian manifold with no conjugate points and $\Sigma = ([\sigma_1],\ldots,[\sigma_k]) \in \pi_1(N)^k$. Then the following hold:

    \vspace{2pt}

    \noindent \textbf{(a)} $N_\Sigma$ is a submanifold of $T^k N$;

    \vspace{2pt}

    \noindent \textbf{(b)} The projection $\pi_\Sigma : N_\Sigma \to N$ is a $C^1$ covering map;

    \vspace{2pt}

    \noindent \textbf{(c)} The fundamental group of $N_\Sigma$ satisfies
    \[
        (\pi_\Sigma)_*(\pi_1(N_\Sigma,v_1 \otimes \cdots \otimes v_k)) = Z([\gamma_{v_1}],\ldots,[\gamma_{v_k}])
    \]
    for each $v_1 \otimes \cdots \otimes v_k \in N_\Sigma$, where $Z$ denotes the centralizer.
\end{theorem}

\begin{proof}
    Parts (a) and (b) have already been proved. Let $\tilde{\alpha} : [a,b] \to N_\Sigma$ be a loop based at $v_1 \otimes \cdots \otimes v_k$, and write $\tilde{\alpha} = \tilde{\alpha}_1 \otimes \cdots \otimes \tilde{\alpha}_k$ and $\alpha = \pi_\Sigma \circ \tilde{\alpha}$. Then $\tilde{\alpha}_i(a) = v_i = \tilde{\alpha}_i(b)$. Define $H_i : [a,b] \times [0,1] \to N$ by $H_i(s,t) = \exp(t\tilde{\alpha}_i(s))$. Since $H_i(a,\cdot) = \gamma_{v_i}(\cdot) = H_i(b,\cdot)$, $\alpha$ must commute with $\gamma_{v_i}$. Conversely, suppose $\alpha : [a,b] \to N$ is a loop based at $\pi_\Sigma(v_1 \otimes \cdots \otimes v_k)$ that commutes with each $\gamma_{v_i}$. Then $\alpha$ lifts via $\pi_\Sigma$ to a path $\tilde{\alpha}$ that begins at $v_1 \otimes \cdots \otimes v_k$. Each map $H_i$, defined as before, is a homotopy connecting $\gamma_{v_i}$ to a geodesic loop $H_i(b,\cdot)$ homotopic to $\gamma_{v_i}(\cdot)$. Since $N$ has no conjugate points, $H_i(b,\cdot) = \gamma_{v_i}(\cdot)$, and therefore $\tilde{\alpha}$ is a loop at $v_1 \otimes \cdots \otimes v_k$.
\end{proof}

\noindent In short, Theorem \ref{structure theorem for geodesic loops} states that $\pi_\Sigma: N_\Sigma \to N$ is a geometric realization of the covering space $\hat{N}/Z([\sigma_1],\ldots,[\sigma_k]) \to N$.

For any $v_1 \otimes \cdots \otimes v_k \in T^k N$, a loop map $\Upsilon_{v_1 \otimes \cdots \otimes v_k} : \tor^k \to N$ will be defined inductively. For a fixed $\theta = (\theta_1,\ldots,\theta_k) \in \tor^k$, identify the $i$-th factor of $\tor^k = S^1 \times \cdots \times S^1$ with $[0,1]/\sim$ in such a way that $\theta_i$ corresponds to $0$, and let $s_i$ be the corresponding closed geodesic. Then $\{[s_1],\ldots,[s_k]\}$ is a minimal generating set for $\pi_1(\tor^k,\theta)$. Define $\upsilon_1 : S^1 \to N$ by $\upsilon_1(s) = \exp(sv_1)$. Suppose $\upsilon_i : \tor^i \to N$ has been defined and satisfies $(\upsilon_i)_*([s_j]) = [\gamma_{v_j}]$ for a fixed $1 \leq i < k$ and $j = 1,\ldots,i$. Then $\upsilon_i$ lifts to a map $\tilde{\upsilon}_i : \tor^i \to N_{([\gamma_{v_1}],\ldots,[\gamma_{v_i}])}$, and one may define $\upsilon_{i+1} : \tor^{i+1} \to N$ by $\upsilon_{i+1}(s_1,\ldots,s_{i+1}) = \exp(s_{i+1}\tilde{\upsilon}_i(s_1,\ldots,s_i))$. Let $\Upsilon_{v_1 \otimes \cdots \otimes v_k} = \upsilon_k$.

These loop maps may be used to construct an explicit homotopy between any two maps $f,g : \tor^k \to N$ that, up to path-conjugation, induce the same homomorphism on $\pi_1(\tor^k)$. Let $u_1 \otimes \cdots \otimes u_k \in T^k N$ be the unique tensor such that $\gamma_{u_i} \in f_*([s_i])$ for each $i$, and let $\Sigma = ([\gamma_{u_1}],\ldots,[\gamma_{u_k}])$. Note that $f$ lifts canonically to a map $\tilde{f} = \tilde{f}_1 \otimes \cdots \otimes \tilde{f}_k: \tor^k \to N_\Sigma$ satisfying $\tilde{f}(\theta) = u_1 \otimes \cdots \otimes u_k$. Write $F_k = f$ and $F_0 = \Upsilon_{u_1 \otimes \cdots \otimes u_k}$. For $1 \leq i \leq k-1$, define $F_i : \tor^k \to N$ by
\[
    F_i(s_1,\ldots,s_k) = \Upsilon_{\tilde{f}_{i+1} \otimes \cdots \otimes \tilde{f}_k(s_1,\ldots,s_i,\theta_{i+1},\ldots,\theta_k)}(s_{i+1},\ldots,s_k)\textrm{.}
\]
Denote by $\varphi : \real^k \to \tor^k$ the universal covering map of $\tor^k$, and fix $\hat{\theta} \in \varphi^{-1}(\theta)$. Then $F_i$ and $F_{i-1}$, in turn, lift to maps $\hat{F}_i$ and $\hat{F}_{i-1}$ from $\real^k$ to $\hat{N}$ that agree at $\hat{\theta}$. The map $\hat{H}_i : [0,1] \times \real^k \to \hat{N}$ defined by
\[
    \hat{H}_i(s,\hat{\vartheta}) = \exp_{\hat{F}_i(\hat{\vartheta})}\big( s\exp^{-1}_{\hat{F}_i(\hat{\vartheta})}(\hat{F}_{i-1}(\hat{\vartheta})) \big)
\]
descends to a homotopy from $F_i$ to $F_{i-1}$. Through concatenation, one obtains a homotopy from $f$ to $\Upsilon_{u_1 \otimes \cdots \otimes u_k}$.

If $\alpha : [0,1] \to N$ is a path as in the statement of Lemma \ref{maps determined by fundamental group}, then $\alpha$ lifts to a path $\tilde{\alpha} : [0,1] \to N_\Sigma$ with $\tilde{\alpha}(0) = u_1 \otimes \cdots \otimes u_k$. As above, one may show that $g$ is homotopic to $\Upsilon_{\tilde{\alpha}(1)}$. The map $\tilde{H} : [0,1] \times \tor^k \to N$ defined by $\tilde{H}(s,\vartheta) = \Upsilon_{\tilde{\alpha}(s)}(\vartheta)$ is a homotopy from $\Upsilon_{u_1 \otimes \cdots \otimes u_k}$ to $\Upsilon_{\tilde{\alpha}(1)}$, which completes the alternate proof of Lemma \ref{maps determined by fundamental group}.

Denote by $C_\Sigma$ the subset of $N_\Sigma$ consisting of tensors containing the initial vectors of closed geodesics, that is,
\[
    C_\Sigma = \{ v_1 \otimes \cdots \otimes v_i \st \gamma_{v_i} \textrm{ is a closed geodesic for each } 1 \leq i \leq k \}\textrm{.}
\]
When $\Sigma = ([\sigma])$, write $N_{[\sigma]}$, $\pi_{[\sigma]}$, and $C_{[\sigma]}$ for $N_\Sigma$, $\pi_\Sigma$, and $C_\Sigma$, respectively.

Each $C_{[\sigma]}$ is the set of initial vectors of closed geodesics $[0,1] \to N$ freely homotopic to $\sigma$, also equal to the image under $\hat{N} \to N_{[\sigma]}$ of the axes of $[\sigma]$. Lemma \ref{closed geodesics have the same length} says that length is constant on $C_{[\sigma]}$, and the first variation formula implies that $C_{[\sigma]}$ is the set of critical points of the length functional on $N_{[\sigma]}$. The following combines two lemmas of Croke--Schroeder \cite{CrokeSchroeder1986}.

\begin{lemma}[(Croke--Schroeder)]\label{croke--schroeder}
    Let $N$ be a compact Riemannian manifold with no conjugate points and $[\sigma] \in \pi_1(N)$. Then $C_{[\sigma]}$ is compact and connected.
\end{lemma}

\noindent When $N$ has no focal points and $\Sigma$ generates a maximal Abelian subgroup of $\pi_1(N)$, the flat torus theorem implies much more about $C_\Sigma$.

\begin{lemma}\label{splitting of spaces of closed geodesics}
    Let $N$ be a complete Riemannian manifold with no focal points and $G$ an Abelian subgroup of $\pi_1(N)$ of rank $m$. Suppose that $G$ is generated by $[\sigma_1],\ldots,[\sigma_k]$, where $k \geq m$, and that $G = Z(G)$. Write $\Sigma = ([\sigma_1],\ldots,[\sigma_k])$. Then the following hold:

    \vspace{2pt}

    \noindent \textbf{(a)} $C_\Sigma$ is isometric to $C \times \tor^m$ for a closed and strongly convex subset $C$ of $N_\Sigma$;

    \vspace{2pt}

    \noindent \textbf{(b)} Each $\tor^m$-fiber of $C_\Sigma \cong C \times \tor^m$ is a totally geodesic, flat, and embedded submanifold of $N_\Sigma$ whose fundamental group has image under $(\pi_\Sigma)_*$ path-conjugate to $G$;

    \vspace{2pt}

    \noindent \textbf{(c)} If $N$ is compact, then $C$ is nonempty and compact.

\end{lemma}

\begin{proof}
    By Theorem \ref{flat torus theorem}(a), $\min(G)$ is isomorphic to $D \times \real^m$ for a closed and strongly convex subset $D$ of $\min(G)$. Let $C$ be the image of $D$ under the covering map $\psi_\Sigma : \hat{N} \to N_\Sigma$, and note that $\min(G) = \psi_\Sigma^{-1}(C_\Sigma)$. By Theorem \ref{structure theorem for geodesic loops}(c), $G$ is naturally identified with the deck transformation group of $\psi_\Sigma$. Since $G$ has rank $m$ and acts on $\min(G)$ by translation in the $\real^m$-factors, $C_\Sigma$ is isomorphic to $D \times (\real^m/G) \cong D \times \tor^m$, and part (b) follows. It is clear that $C$ is closed and, since $\min(G)$ is strongly convex and $\psi_\Sigma$ is injective on $D$, strongly convex. This proves (a). Part (c) follows immediately from Theorem \ref{flat torus theorem}(d) and Lemma \ref{croke--schroeder}.
\end{proof}

\begin{remark}\label{loop map is totally geodesic}
    Suppose the assumptions of Lemma \ref{splitting of spaces of closed geodesics} hold. Then a map $f : \tor^k \to N$ such that $f \circ s_i \in [\sigma_i]$ for each $1 \leq i \leq k$ is totally geodesic if and only if it is of the form $f = \Upsilon_{v_1 \otimes \cdots \otimes v_k}$ for some $v_1 \otimes \cdots \otimes v_k \in C_\Sigma$, in which case each $v_i$ is the unique vector such that $\gamma_{v_i} \in [\sigma_i]$.
\end{remark}

\subsection{Center of mass}

The Riemannian center of mass will be used to average maps in the universal cover of a manifold with no focal points. Elementary results about the center of mass are widely known, although they're developed somewhat differently here than in \cite{GroveKarcher1973}, the source of its modern renaissance.

Note that, for a fixed $q \in N$ and all $x \in B(q,\rho(q)/2)$, the function $d^2(\cdot,x)$ is strictly convex on $B(q,\rho(q)/2)$. The following may be proved using the first variation formula.

\begin{lemma}\label{distance strictly decreasing}
    Let $N$ be a complete Riemannian manifold, $q \in N$, and $Y \subseteq B(q,\rho(q)/2)$ a closed and convex set. For any $x \in B(q,\rho(q)/2) \setminus Y$, let $\gamma : [a,b] \to B(q,\rho(q)/2)$ be a geodesic from $x$ to $Y$ with length equal to $d(x,Y)$. Then, for each $y \in Y$, the function $t \mapsto d(\gamma(t),y)$ is strictly decreasing.
\end{lemma}

\noindent A \textbf{mass distribution} is a measurable function $m : Z \to N$, where $(Z,\mu)$ is a measure space of total measure one. Whenever $m$ maps into a ball $B(q,R)$ of radius $R < \rho(q)/2$, the function $x \mapsto \int_Z d^2(x,m(z))\,d\mu$ is strictly convex on the closed ball $\overline{B(q,R)}$ and, consequently, attains a unique minimum $\Phi_m$ therein. The point $\Phi_m$ is the \textbf{center of mass} of $m$. By Lemma \ref{distance strictly decreasing}, $\Phi_m$ lies inside the convex closure of $m(Z)$ and is the unique minimum of $x \mapsto \int_Z d^2(x,m(z))\,d\mu$ within $B(q,\rho(q)/2)$. If $R < \rho(q)/6$, the triangle inequality implies that $\Phi_m$ is the unique minimum on $N$.

\begin{lemma}\label{center of mass commutes with isometries}
    Let $N$ be a complete Riemannian manifold, $q \in N$, and $m : Z \to N$ a mass distribution. Suppose that $m(Z) \subseteq B(q,R)$ for some $R < \rho(q)/2$. Then, for any isometry $\alpha$ of $N$, $\alpha(\Phi_m) = \Phi_{\alpha \circ m}$.
\end{lemma}

\begin{lemma}\label{center of mass in a product space}
    Let $N$ be a complete Riemannian manifold, $q \in N$, and $m : Z \to N$ a mass distribution. Suppose that $m(Z) \subseteq S$ for some closed and convex subset of $B(q,R)$, where $R < \rho(q)/2$. If $S$ is isometric to $S_1 \times \cdots \times S_k$, where the $S_i$ are convex subsets of $S$, then $\Phi_m = (\Phi_{\pi_1 \circ m},\ldots,\Phi_{\pi_k \circ m})$, where $\pi_i$ is projection onto $S_i$.
\end{lemma}

\noindent The center of mass will be used when $Z$ is a finite set. If $\Lambda = (\lambda_1,\ldots,\lambda_k) \in [0,1]^k$ satisfies $\sum_i \lambda_i = 1$ and $y = (y_1,\ldots,y_k) \in S^k$, where $S$ is a subset of $B(q,\rho(q)/2)$ for some $q \in N$, then $\Phi_\Lambda(y_1,\ldots,y_k)$ will denote $\Phi_m$, where $m : \{ 1,\ldots,k \} \to N$ is the mass distribution satisfying $m(i) = y_i$ and $\mu(i) = \lambda_i$.

\begin{remark}\label{grove--karcher}
    In \cite{GroveKarcher1973}, $\Phi_m$ is defined as the unique minimum of $x \mapsto \int_Z \exp_x^{-1}(m(z))\,d\mu$ within a sufficiently small ball $B(q,R)$. Whenever $R < \rho(q)/2$, this is equivalent to the definition given here.
\end{remark}

\subsection{Commutative diagrams}

Suppose $M_0$ and $M_1$ are connected $C^1$ manifolds, $M_0$, $M_0 \times \real^k$, and $M_1 \times \tor^k$ have Riemannian metrics, where $M_0 \times \real^k$ has a product metric with a flat $\real^k$-factor, $\pi_0 : M_0 \times \real^k \to \real^k$ and $\pi_1 : M_1 \times \tor^k \to \tor^k$ are the projections onto the second components, and $\psi : M_0 \times \real^k \to M_1 \times \tor^k$ and $\phi : \real^k \to \tor^k$ are covering maps. The \textbf{NNRC diagram}
\begin{equation}\label{diagram1}
    \begin{gathered}
        \xymatrix{ M_0 \times \real^k \ar[r]^-{\pi_0} \ar[d]^-\psi & \real^k \ar[d]^-\phi \\ M_1 \times \tor^k \ar[r]^-{\pi_1} & \tor^k}
    \end{gathered}
\end{equation}
\textbf{commutes isometrically} if it commutes and $\psi$ and $\phi$ are local isometries. These diagrams were introduced in \cite{Dibble2018a} and are motivated by the Cheeger--Gromoll splitting theorem \cite{CheegerGromoll1972,CheegerGromoll1971}, a consequence of which is that every compact manifold with nonnegative Ricci curvature is finitely covered by a diffeomorphic product in a NNRC diagram that commutes isometrically.

\begin{theorem}[(Cheeger--Gromoll)]\label{cheeger--gromoll}
    If $M$ is compact and has nonnegative Ricci curvature, then it is finitely and locally isometrically covered by a manifold $M_1 \times \tor^k$ in a diagram of the form \eqref{diagram1} that commutes isometrically, in which $M_0$ is compact and simply connected.
\end{theorem}

\noindent The following collects properties of NNRC diagrams developed in \cite{Dibble2018a}.

\begin{lemma}\label{nnrc diagrams}
    Suppose the diagram \eqref{diagram1} commutes isometrically. Then the following hold:

    \vspace{2pt}

    \noindent \textbf{(a)} If $M_1 \times \tor^k$ has finite volume, every geodesic in the Riemannian universal covering space $\hat{N}$ is minimal, and $f : M_1 \times \tor^k \to N$ is continuous and acts trivially on $\pi_1(M_1)$, then $f$ is totally geodesic if and only if $f$ is constant along each $M_1$-fiber and, with respect to the flat metric on $\tor^k$, totally geodesic along each $\tor^k$-fiber.

    \vspace{2pt}

    \noindent \textbf{(b)} If $\Gamma$ is the deck transformation group of $\psi$ and $\mathscr{I}(M_0)$ and $\mathscr{I}(\real^k)$ are the isometry groups of $M_0$ and $\real^k$, respectively, then $\Gamma \subseteq \mathscr{I}(M_0) \times \mathscr{I}(\real^k)$.
\end{lemma}

\noindent A well-known theorem of Bieberbach states that every compact flat manifold is finitely and normally covered by a flat torus. Theorem \ref{cheeger--gromoll} and the following prove a more general statement for manifolds with nonnegative Ricci curvature.

\begin{lemma}\label{normal covering}
    Suppose the diagram \eqref{diagram1} commutes isometrically and $\psi_1 : M_1 \times \tor^k \to M$ is a finite covering map. Then there exist manifolds $\tilde{M}_0$ and $\tilde{M}_1$, Riemannian metrics on $\tilde{M}_0$ and $\tilde{M}_1 \times \tor^k$, covering maps $\tilde{\psi}$ and $\tilde{\phi}$, and finite and normal covering maps $\tilde{\psi}_1$, $\zeta_1$, $\xi_0$, and $\xi_1$ such that the diagram
    \[
        \begin{gathered}
            \xymatrix{ \tilde{M}_0 \times \real^k \ar[r]^-{\tilde{\pi}_0} \ar[d]^-{\tilde{\psi}} & \real^k \ar[d]^-{\tilde{\phi}} \\ \tilde{M}_1 \times \tor^k \ar[r]^-{\tilde{\pi}_1} & \tor^k}
        \end{gathered}
    \]
    commutes isometrically, the diagram\\
    \begin{center}
    \begin{tikzcd}[row sep=scriptsize, column sep=scriptsize]
        & M_0 \times \real^k \arrow[rr, "\pi_0"] \arrow[dd, "\psi" near start] & & \real^k \arrow[dd, "\phi"] \\
        \tilde{M}_0 \times \real^k \arrow[ur, "\xi_0 \times \mathop{id}"] \arrow[rr, crossing over, "\tilde{\pi}_0" near end] \arrow[dd, "\tilde{\psi}"] & & \real^k \arrow[ur, "\mathop{id}"] \\
        & M_1 \times \tor^k \arrow[rr, "\pi_1" near start] \arrow[ddd, "\psi_1" near end] & & \tor^k \\
        \tilde{M}_1 \times \tor^k \arrow[ur, "\xi_1 \times \zeta_1"] \arrow[rr, "\tilde{\pi}_1" near end, crossing over] \arrow[ddr, "\tilde{\psi}_1"] & & \tor^k \arrow[ur, "\zeta_1"] \arrow[from=uu, crossing over, "\tilde{\phi}" near start] \\
        & & & \\
        & M & &
    \end{tikzcd}\\
    \end{center}
    commutes, and $\zeta_1$, $\xi_0$, and $\xi_1 \times \zeta_1$ are local isometries.
\end{lemma}

\noindent The proof of Lemma \ref{normal covering} uses the following basic algebraic facts, which are consequences of the first and, respectively, second isomorphism theorems.

\begin{lemma}\label{first algebraic lemma}
    Let $H$ be a finite-index subgroup of a group $G$. Then $H$ contains a subgroup $I$ that is normal in $G$ and has index satisfying $[G : I] \leq [G : H]!$.
\end{lemma}

\begin{lemma}\label{second algebraic lemma}
    Let $G_1$ and $G_2$ be groups and $H$ a finite-index normal subgroup of $G_1 \times G_2$. Then there exist normal subgroups $H_i$ of $G_i$, $i = 1,2$, such that $H_1 \times H_2 \subseteq H$, $[G_i : H_i] \leq [G_1 \times G_2 : H]$, and $[G_1 \times G_2 : H_1 \times H_2] \leq [G_1 \times G_2 : H]^2$.
\end{lemma}

\begin{proof}[of Lemma \ref{normal covering}]
    Fix $(\tilde{p},\tilde{x}) \in M_1 \times \tor^k$ and $(\hat{p},\hat{x}) \in \psi^{-1}(\tilde{p},\tilde{x})$ as basepoints for $\pi_1(M_1 \times \tor^k)$ and $\pi_1(M_0 \times \real^k)$, respectively. Write $G = (\psi_1)_*(\pi_1(M_1 \times \tor^k))$, $G_1 = (\psi_1 \circ \iota_{\tilde{x}})_*(\pi_1(M_1))$, and $G_2 = (\psi_1 \circ \iota_{\tilde{p}})_*(\pi_1(\tor^k))$, so that $G \cong G_1 \times G_2$. By Lemma \ref{first algebraic lemma}, $\pi_1(M)$ has a finite-index normal subgroup $H$ contained in $G$. By Lemma \ref{second algebraic lemma}, there exist finite-index normal subgroups $H_i$ of $G_i$ such that $H_1 \times H_2$ is a finite-index subgroup of $G$ contained in $H$. Thus $\tilde{H}_1 = (\psi_1 \circ \iota_{\tilde{x}})_*^{-1}(H_1)$ and $\tilde{H}_2 = (\psi_1 \circ \iota_{\tilde{p}})_*^{-1}(H_2)$ are finite-index normal subgroups of $G_1$ and $G_2$, respectively. It follows that there exist a manifold $\tilde{M}_1$, $\tilde{p}_1 \in \tilde{M}_1$, $\tilde{x}_1 \in \tor^k$, and finite and normal covering maps $\xi_1 : \tilde{M}_1 \to M_1$ and $\zeta_1 : \tor^k \to \tor^k$ such that $(\xi_1)_*(\pi_1(\tilde{M}_1)) = \tilde{H}_1$, $\xi_1(\tilde{p}_1) = \tilde{p}$, $\zeta_1(\tilde{x}_1) = \tilde{x}$, and $(\zeta_1)_*(\pi_1(\tor^k)) = \tilde{H}_2$. By construction, $\tilde{\psi}_1 = \psi_1 \circ (\xi_1 \circ \zeta_1)$ satisfies $(\tilde{\psi}_1)_* (\pi_1(\tilde{M}_1 \times \tor^k)) = H_1 \times H_2$ and, consequently, is finite and normal.

    By general theory, there exists a covering map $\tilde{\phi} : \real^k \to \tor^k$ such that $\phi = \zeta_1 \circ \tilde{\phi}$ and $\tilde{\phi}(\hat{x}) = \tilde{x}_1$. Note that $\chi = \rho_1 \circ \psi \circ \iota_{\hat{x}} : M_0 \to M_1$ is a covering map satisfying $\chi(\hat{p}) = \tilde{p}$, where $\rho_1 : M_1 \times \tor^k \to M_1$ is projection. Let $I = \tilde{H}_1 \cap \chi_*(\pi_1(M_0))$. Then there exist a manifold $\tilde{M}_0$, $\tilde{p}_0 \in \tilde{M}_0$, and covering maps $\tilde{\chi} : \tilde{M}_0 \to \tilde{M}_1$ and $\xi_0 : \tilde{M}_0 \to M_0$ such that $\xi_1 \circ \tilde{\chi} = \chi \circ \xi_0$, $\tilde{\chi}(\tilde{p}_0) = \tilde{p}_1$, $\xi_0(\tilde{p}_0) = \hat{p}$, and $(\xi_1 \circ \tilde{\chi})_*(\pi_1(\tilde{M}_0)) = I = (\chi \circ \xi_0)_*(\pi_1(\tilde{M}_0))$. Since $I$ is a normal and, by the second isomorphism theorem, finite-index subgroup of $\chi_*(\pi_1(M_0))$, $\xi_0$ is finite and normal.

    Note that the diagram
    \begin{equation}\label{diagram2}
        \xymatrix{ \hat{M}_0 \ar[r]^-{\xi_0} \ar[d]^-{\hat{\chi}} & M_0 \ar[r]^-{\iota_{\overline{x}}} \ar[d]^-\chi & M_0 \times N_0 \ar[d]^-{\psi} \\ \hat{M}_1 \ar[r]^-{\xi_1} & M_1 \ar[r]^-{\iota_{\tilde{x}}} & M_1 \times N_1 }
    \end{equation}
    commutes. It follows that $(\psi \circ \iota_{\tilde{x}} \circ \xi_0)_*(\pi_1(\tilde{M}_0))$ is a subgroup of $(\iota_{\tilde{x}} \circ \xi_1)_*(\pi_1(\tilde{M}_1))$ and, consequently, that $(\psi_0 \circ (\xi_0 \times \mathop{id}))_*(\pi_1(\tilde{M}_0 \times \real^k))$ is a subgroup of $(\xi_1 \times \zeta_1)_*(\pi_1(\tilde{M}_1 \times \tor^k))$. Therefore, there exists a covering map $\tilde{\psi} : \tilde{M}_0 \times \real^k \to \tilde{M}_1 \times \tor^k$ such that $(\xi_1 \times \zeta_1) \circ \tilde{\psi} = \psi \circ (\xi_0 \times \mathop{id})$ and $\tilde{\psi}(\tilde{p}_0,\tilde{x}_0) = (\tilde{p}_1,\tilde{x}_1)$.

    By construction, $\zeta_1 \circ \tilde{\pi}_1 \circ \tilde{\psi} = \zeta_1 \circ \tilde{\phi} \circ \tilde{\pi}_0$. Since $\tilde{\pi}_1 \circ \tilde{\psi}$ and $\tilde{\phi} \circ \tilde{\pi}_0$ agree at $(\tilde{p}_0,\tilde{x}_0)$, they must be the same function. Thus the second diagram commutes and, when $\tilde{M}_0$, $\tilde{M}_1 \times \tor^k$, and the upper $\tor^k$ are endowed with the pull-back metrics that make $\xi_0$, $\xi_1 \times \zeta_1$, and $\zeta_1$ local isometries, the first diagram commutes isometrically.
\end{proof}

\section{Proof of the main theorems}

When $M$ is a compact Riemannian manifold with nonnegative Ricci curvature, Theorem \ref{cheeger--gromoll} states that $M$ is finitely and locally isometrically covered by a manifold $M_1 \times \tor^k$ in a NNRC diagram \eqref{diagram1} that commutes isometrically and in which $M_0$ is compact and simply connected. It follows that $M_1$ is compact, each covering map $\chi = \rho_1 \circ \psi \circ \iota_{\hat{x}}$ in diagram \eqref{diagram2} is finite, and, consequently, that $\pi_1(M_1)$ is finite. If $N$ has no focal points, then $\pi_1(N)$ is torsion free, so every map $M_1 \times \tor^k \to N$ acts trivially on the fundamental group of $M_1$. Thus Theorems \ref{main theorem for nnrc} and \ref{product version for nnrc} are special cases of the following results.

\begin{theorem}\label{main theorem}
    Let $M$ be a compact Riemannian manifold, $\psi_1 : M_1 \times \tor^k \to M$ a finite covering map, where $M_1 \times \tor^k$ appears in a NNRC diagram \eqref{diagram1} that commutes isometrically, $N$ a complete Riemannian manifold with no focal points, and $[F]$ a homotopy class of maps from $M$ to $N$ such that $[F \circ \psi_1]$ acts trivially on $\pi_1(M_1)$. Then the following hold:\\
    \textbf{(a)} The set of totally geodesic maps in $[F]$ is path-connected;\\
    \textbf{(b)} If $[F]$ contains a totally geodesic map, then a map in $[F]$ is energy-minimizing if and only if it is totally geodesic;\\
    \textbf{(c)} If $N$ is compact, then $[F]$ contains a totally geodesic map.
\end{theorem}

\begin{theorem}\label{product version}
    Let $M$ be a compact Riemannian manifold, $\psi_1 : M_1 \times \tor^k \to M$ a finite covering map, where $M_1 \times \tor^k$ appears in a NNRC diagram \eqref{diagram1} that commutes isometrically, $N$ a compact Riemannian manifold with no focal points, $W$ a manifold, and $f : W \times M \to N$ a continuous map such that $f \circ (\mathop{id} \times \psi_1) : W \times M_1 \times \tor^k \to N$ acts trivially on $\pi_1(M_1)$.  Then $f$ is homotopic to a map that's totally geodesic on each $M$-fiber.
\end{theorem}

\noindent The proof of Theorems \ref{main theorem} and \ref{product version} uses in a crucial way a center-of-mass gluing technique developed by Cao--Cheeger--Rong \cite{CaoCheegerRong2004} in the context of nonpositive curvature. The following extends the key idea to the case of no focal points. Note that, when $N$ has no focal points, $\rho(\hat{N}) = \infty$, so the center of mass $\Phi_m$ is defined for any mass distribution $m : Z \to \hat{N}$.

\begin{lemma}\label{gluing technique for toruses}
    Let $N$ be a complete Riemannian manifold with no focal points, $h_1,\ldots,h_m : [a,b] \times \tor^k \to N$ continuous functions such that $h_i(a,\cdot) = h_j(a,\cdot)$ for all $i,j$, and $\Lambda = (\lambda_1,\ldots,\lambda_m) \in [0,1]^m$ such that $\sum_i \lambda_i = 1$. Then there exists a continuous function $h : [a,b] \times \tor^k \to N$ characterized by the following property:

    \vspace{2pt}

    \begin{center}
        (*) If $\hat{h}_1,\ldots,\hat{h}_m : [a,b] \times \real^k \to \hat{N}$ are lifts of $h_1,\ldots,h_m$ that agree on $\{ a \} \times \real^k$, then $h(s,\vartheta) = \pi \circ \Phi_\Lambda(\hat{h}_1(s,\hat{\vartheta}),\ldots,\hat{h}_m(s,\hat{\vartheta}))$ whenever $s \in [a,b]$ and $\hat{\vartheta} \in \pi^{-1}(\vartheta)$.
    \end{center}

    \vspace{2pt}

    \noindent Moreover, if each $h_i(b,\cdot)$ is totally geodesic, then $h(b,\cdot)$ is totally geodesic.
\end{lemma}

\begin{proof}
     According to Lemma \ref{center of mass commutes with isometries}, $h$ is well defined by (*). Suppose each $h_i(b,\cdot)$ is totally geodesic. Let $\hat{h}_i$ be lifts of $h_i$ as in (*), and, for each $\ell = 1,\ldots,k$, let $\sigma_\ell$ be the loop $h_i(a,s_\ell(\cdot))$. Then $\hat{h}_i(\{ b \} \times \real^k) \subseteq \min(G)$, where $G$ is the subgroup of $\pi_1(N)$ generated by $[\sigma_1],\ldots,[\sigma_k]$. With respect to the isometric splitting of $\min(G)$ as $D \times \real^k$ in Theorem \ref{flat torus theorem}, in which $D$ is closed and strongly convex, each $\hat{h}_i(b,\cdot)$ must be of the form $\hat{h}_i(b,\cdot) = (d_i,\hat{H}_i(\cdot))$ for $d_i \in D$ and totally geodesic $\hat{H}_i : \real^k \to \real^k$. Note that $\Phi_\Lambda(d_1,\ldots,d_k) \in D$. By Lemma \ref{center of mass in a product space}, $\Phi_\Lambda(\hat{h}_1(b,\hat{\vartheta}),\ldots,\hat{h}_m(b,\hat{\vartheta})) = \big( \Phi_\Lambda(d_1,\ldots,d_k),\Phi_\Lambda(\hat{H}_1(\hat{\vartheta}),\ldots,\hat{H}_k(\hat{\vartheta})) \big)$. In $\real^k$, the center of mass is the usual weighted average, so $\hat{\vartheta} \mapsto \Phi_\Lambda(\hat{H}_1(\hat{\vartheta}),\ldots,\hat{H}_k(\hat{\vartheta}))$ is totally geodesic. The result follows from (*).
\end{proof}

\noindent Using Lemma \ref{gluing technique for toruses} and a partition of unity, one may extend this gluing technique to maps defined on products $W \times M_1 \times \tor^k$.

\begin{lemma}\label{glue together locally defined homotopies}
    Let $N$ be a compact Riemannian manifold with no focal points, $W$ and $M_1$ connected manifolds, and $f : W \times M_1 \times \tor^k \to N$ a continuous function that acts trivially on $\pi_1(M_1)$. Then there exists a continuous function $H : [0,1] \times W \times M_1 \times \tor^k \to N$ such that $H(0,\cdot) = f(\cdot)$ and $H(1,\cdot)$ is constant along each $M_1$-fiber and totally geodesic along each $\tor^k$-fiber.
\end{lemma}

\begin{proof}
    Fix ${p} \in W$ and $x_0 \in M_1$, and let $V_{{p}} \subseteq W$ be any contractible neighborhood of ${p}$. Then $f|_{V_{{p}} \times M_1 \times \tor^k}$ lifts to a map $\hat{f}_{{p}} : V_{{p}} \times M_1 \times \real^k \to \hat{N}$, and for each such lift the map $\hat{h}_{{p}} : [0,1] \times V_{{p}} \times M_1 \times \real^k \to \hat{N}$ defined by
    \[
        \hat{h}_{{p}}(s,w,x,\hat{\vartheta}) = \exp_{\hat{f}_{{p}}(w,x,\hat{\vartheta})}\big( s\exp_{\hat{f}_{{p}}(w,x,\hat{\vartheta})}^{-1}(\hat{f}_{{p}}({p},x_0,\hat{\vartheta})) \big)
    \]
    descends to a homotopy from $f|_{V_{{p}} \times M_1 \times \tor^k}$ to a map $F_{{p}} : V_{{p}} \times M_1 \times \tor^k \to N$ that's constant along each $(V_{{p}} \times M_1)$-fiber. For each $i = 1,\ldots,k$, denote by $\sigma_i(\cdot)$ the loop $f({p},x_0,s_i(\cdot))$ and by $u_i$ the unique vector such that $\gamma_{u_i} \in [\sigma_i]$. By Lemma \ref{maps determined by fundamental group}, $F_{{p}}$ is homotopic to a map $G_{{p}}$ that's constant along each $(V_{{p}} \times M_1)$-fiber and agrees with $\Upsilon_{u_1 \otimes \cdots \otimes u_k}$ along each $\tor^k$-fiber. Let $\Sigma = ([\sigma_1],\ldots,[\sigma_k]) \in \pi_1(N)^k$. By Lemma \ref{splitting of spaces of closed geodesics}, $C_\Sigma$ is nonempty. For any path $\tilde{\alpha} : [0,1] \to N_\Sigma$ from $u_1 \otimes \cdots \otimes u_k$ to $C_\Sigma$, the map $(s,w,m,\vartheta) \mapsto \Upsilon_{\tilde{\alpha}(s)}(\vartheta)$ for $s \in [0,1]$ is a homotopy from $G_{{p}}$ to a map that's constant along each $(V_{{p}} \times M_1)$-fiber and totally geodesic along each $\tor^k$-fiber.

    Therefore, around each $w \in W$, there exists a homotopy $h_w : [0,1] \times V_w \times M_1 \times \tor^k \to N$ from $f|_{V_w \times M_1 \times \tor^k}$ to a map that's constant along each $(V_w \times M_1)$-fiber and totally geodesic along each $\tor^k$-fiber. Choose open sets $U_w \subset \overline{U}_w \subset V_w$, and let $\{ \lambda_w \}$ be a partition of unity subordinate to the open cover $\{ U_w \}$ of $W$. For each ${p} \in W$, let $W_{{p}}$ be a neighborhood of ${p}$ on which only finitely many $\lambda_w$, say $\lambda_{w_1},\ldots,\lambda_{w_m}$, are nonzero. Without loss of generality, one may suppose that ${p}$ is in the support of each $\lambda_{w_i}$ and, consequently, that ${p} \in V_{w_i}$. One may also shrink $W_{{p}}$, if necessary, so that each $h_{w_i}$ is defined on $[a,b] \times W_{{p}} \times M \times \tor^k$. In this case, $\Lambda_{{p}} = (\lambda_{w_1},\ldots,\lambda_{w_m})$ is a continuous function from $W_{{p}}$ into $[0,1]^m$ such that $\sum_{i=1}^k \lambda_{w_i} = 1$.

    For each $w \in W_{{p}}$ and $x \in M$, let $H_{{p},w,x} : [a,b] \times \{ w \} \times \{ x \} \times \tor^k \to N$ be the map satisfying (*) in Lemma \ref{gluing technique for toruses}, where $\Lambda = \Lambda(w)$ and $h_i = h_{w_i}|_{\{ w \} \times \{ x \} \times \tor^k}$. Then $H_{{p},w,x}(b,\cdot)$ is totally geodesic along each $\tor^k$-fiber. Define $H_{{p}} : [a,b] \times W_p \times M \times \tor^k \to N$ by $H_{{p}}(s,w,x,\vartheta) = H_{{p},w,x}(s,\vartheta)$. Since $\Phi_\Lambda(\hat{x}_1,\ldots,\hat{x}_m)$ varies continuously with $\Lambda$ and the $\hat{x}_i$, $H_{{p}}$ is continuous. It is routine to check that setting $H(s,w,x,\vartheta) = H_{{p}}(s,w,x,\vartheta)$ for any ${p}$ such that $w \in W_{{p}}$ yields a well-defined function $H : [a,b] \times W \times M \times \tor^k \to N$ with the desired properties.
\end{proof}

\noindent When $M$ is finitely and normally covered by a manifold $M_1 \times \tor^k$ in an NNRC diagram \eqref{diagram1} that commutes isometrically, maps $W \times M_1 \times \tor^k \to N$ that are totally geodesic on the $(M_1 \times \tor^k)$-fibers may be used to construct maps $W \times M \to N$ that are totally geodesic on the $M$-fibers. The idea is to average over the deck transformation group of $M_1 \times \tor^k \to M$ and apply Lemma \ref{nnrc diagrams}. Combining this with Lemma \ref{normal covering} proves Theorem \ref{product version}.

\begin{proof}[of Theorem \ref{product version}]
    By Lemma \ref{normal covering}, one may suppose without loss of generality that $\psi_1$ is normal. Apply Lemma \ref{glue together locally defined homotopies} to produce a homotopy $\tilde{H} : [0,1] \times W \times M_1 \times \tor^k \to N$ from $\tilde{f} = f \circ (\mathop{id} \times \psi_1)$ to a map $\tilde{F}$ that's constant along each $M_1$-fiber and totally geodesic along each $\tor^k$-fiber. By Lemma \ref{nnrc diagrams}(a), $\tilde{F}$ is totally geodesic along each $(M_1 \times \tor^k)$-fiber. Lift $\tilde{H}$ to a homotopy $\hat{H} : [0,1] \times W \times M_0 \times \real^k \to \hat{N}$ that ends in a map $\hat{F}$ that's totally geodesic along each $(M_0 \times \real^k)$-fiber.

    Let $\Lambda = (1/m,\ldots,1/m) \in [0,1]^m$, where $m$ is the number of sheets of $\psi_1$. Let $\Gamma = \{ \gamma_1,\ldots,\gamma_m \}$ be the deck transformation group of $\psi_1$, where each $\gamma_i$ is identified with an element of $\pi_1(M)$, and let $\xi_i$ denote the deck transformation of $\pi$ corresponding to $f_*(\gamma_i)$. By Lemma \ref{nnrc diagrams}(b), there exist deck transformations $\hat{\gamma}_i$ of $\psi_1 \circ \psi$, each of which splits as $\hat{\gamma}_i = \hat{\alpha}_i \times \hat{\beta}_i$ for isometries $\hat{\alpha}_i$ of $M_0$ and $\hat{\beta}_i$ of $\real^k$, such that each diagram
    \[
        \xymatrix{ M_0 \times \real^k \ar[r]^-{\hat{\gamma}_i} \ar[d]^-\psi & M_0 \times \real^k \ar[d]^-\psi \\ M_1 \times \tor^k \ar[r]^-{\tilde{\gamma}_i} & M_1 \times \tor^k }
    \]
    commutes. Define maps $\tilde{h}_i : [0,1] \times W \times M_1 \times \tor^k \to N$ by
    \[
        \tilde{h}_i(s,w,\tilde{x}) = \pi \circ \xi_i^{-1} \circ \hat{H}(s,w,\hat{\gamma}_i(\hat{x}))
    \]
    for any $\hat{x} \in \psi^{-1}(\tilde{x})$. It is routine to check that each $\tilde{h}_i$ is well defined and satisfies $\tilde{h}_i(0,\cdot) = \tilde{f}(\cdot)$. Since each $\hat{H}(1,\cdot) = \hat{F}(\cdot)$ is totally geodesic along each $(M_0 \times \real^k)$-fiber, so is each $\tilde{h}_i(1,\cdot)$. Since $\hat{F}$ is constant along each $M_0$-fiber, the splittings $\hat{\gamma}_i = \hat{\alpha}_i \times \hat{\beta}_i$ ensure that each $\tilde{h}_i(1,\cdot)$ is constant along each $M_1$-fiber.

    Define a map $\tilde{h} : [0,1] \times W \times M_1 \times \tor^k \to N$ by
    \[
        \tilde{h}(s,w,\tilde{x}) = \pi \circ \Phi_\Lambda \big( \xi_1^{-1} \circ \hat{H}(s,w,\hat{\gamma}_1(\hat{x})),\ldots,\xi_m^{-1} \circ \hat{H}(s,w,\hat{\gamma}_m(\hat{x})) \big)
    \]
    for any $\hat{x} \in \psi^{-1}(\tilde{x})$. Then $\tilde{h}(0,\cdot) = \tilde{f}(\cdot)$, $\tilde{h}(1,\cdot)$ is constant along each $M_1$-fiber, and, by Lemma \ref{gluing technique for toruses}, $\tilde{h}(1,\cdot)$ is totally geodesic along each $\tor^k$-fiber. So $\tilde{h}(1,\cdot)$ is totally geodesic along each $(M_1 \times \tor^k)$-fiber.

    Note that $\tilde{h}(\cdot)$ is invariant under the action of $\Gamma$ on $M_1 \times \tor^k$. Since $\psi_1$ is normal, $\tilde{h}$ descends to a homotopy $h : [0,1] \times W \times M \to N$ that initially agrees with $f$. By Lemma \ref{nnrc diagrams}(a), $h$ ends at a map that's totally geodesic along each $M$-fiber.
\end{proof}

\begin{proof}[of Theorem \ref{main theorem}]
    Part (c) is the special case of Theorem \ref{product version} in which $W$ is a point. Part (b) is an immediate consequence of Theorem 1.3(a) in \cite{Dibble2018a}. It remains to prove (a). By Lemma \ref{normal covering}, one may take $\psi_1$ to be normal. Let $f$ and $g$ be totally geodesic maps in $[F]$. Write $\tilde{F} = F \circ \psi_1$, and, for a fixed $\tilde{x}_0 \in M_1$ and $\theta \in \tor^k$, let $s_1,\ldots,s_k$ be the standard generators for $\pi_1(\tor^k)$ as in the proof of Lemma \ref{maps determined by fundamental group}. Write $\sigma_i = \tilde{F} \circ \iota_{\tilde{x}_0} \circ s_i$, i.e., $\sigma_i(\cdot) = \tilde{F}(\tilde{x}_0,s_i(\cdot))$, and $\Sigma = ([\sigma_1],\ldots,[\sigma_k])$. By assumption, $G = \tilde{F}_*(\pi_1(M_1 \times \tor^k))$ is an Abelian group generated by $[\sigma_1],\ldots,[\sigma_k]$. Note that $\min(G) = \psi_\Sigma^{-1}(C_\Sigma)$. Since $\tilde{f} = f \circ \psi_1$ and $\tilde{g} = g \circ \psi_1$ are totally geodesic, Lemma \ref{nnrc diagrams}(a) implies that $\tilde{f} = \Upsilon_{v_1 \otimes \cdots \otimes v_k} \circ \pi_1$ and $\tilde{g} = \Upsilon_{u_1 \otimes \cdots \otimes u_k} \circ \pi_1$ for $v_1 \otimes \cdots \otimes v_k,u_1 \otimes \cdots \otimes u_k \in C_\Sigma$. By Theorem \ref{flat torus theorem}(a), there exists a minimal geodesic $\alpha : [0,1] \to N_\Sigma$ connecting $u_1 \otimes \cdots \otimes u_k$ to $v_1 \otimes \cdots \otimes v_k$ whose image lies in $C_\Sigma$. The map $(s,\tilde{x},\vartheta) \mapsto \Upsilon_{\alpha(s)} \circ \pi_1(\tilde{x},\vartheta) = \Upsilon_{\alpha(s)}(\vartheta)$ is a homotopy from $\tilde{f}$ to $\tilde{g}$ through totally geodesic maps. Since $\psi_1$ is normal, averaging over its deck transformation group, as in the proof of Theorem \ref{product version}, yields a homotopy from $f$ to $g$ through totally geodesic maps.
\end{proof}

\begin{acknowledgements}\label{ackref}
    These results appear in my doctoral dissertation \cite{Dibble2014}. They were greatly improved through many helpful conversations with my dissertation advisor, Xiaochun Rong, as well as Penny Smith, Jason DeBlois, Christopher Croke, and Armin Schikorra. An earlier proof of these results used the work of Karuwannapatana--Maneesawarng \cite{KaruwannapatanaManeesawarng2007}, which I was introduced to by Stephanie Alexander.
\end{acknowledgements}

\bibliographystyle{amsplain}
\bibliography{bibliography}

\affiliationone{
   James Dibble\\
   Department of Mathematics\\
   University of Iowa\\
   14 MacLean Hall\\
   Iowa City, IA 52242\\
   United States of America
   \email{james-dibble@uiowa.edu}}

\end{document}